\documentclass[12pt,reqno]{amsart}
\usepackage{etoolbox}

\makeatletter
\patchcmd\maketitle
{\uppercasenonmath\shorttitle}
{}
{}{}
\patchcmd\maketitle
{\@nx\MakeUppercase{\the\toks@}}
{\the\toks@}
{}
{}{}
\patchcmd\@settitle{\uppercasenonmath\@title}{\Large}{}{}
\patchcmd\@setauthors
{\MakeUppercase{\authors}}
{\authors}
{}{}
\makeatother
\usepackage{amsmath,amssymb,amsthm, amsfonts}
\usepackage{color}
\usepackage{url}
\usepackage{tikz-cd}
\usepackage{combelow}
\input{mathrsfs.sty}
\usepackage[utf8]{inputenc}
\usepackage[T1]{fontenc}
\newtheorem{theorem}{Theorem}[section]

\newtheorem{remark}{Remark}[section]

\numberwithin{equation}{section}

\newcommand\norm[1]{\left\lVert#1\right\rVert}
\newcommand\aps[1]{\left\lvert#1\right\rvert}
\newcommand\skal[2]{\left\langle #1,#2\right\rangle}

\usepackage[colorlinks=true]{hyperref}
\hypersetup{urlcolor=blue, citecolor=red , linkcolor= blue}

\begin{document}
		\title[Birkhoff-James orthogonality: A Minimax Theorem approach]{Birkhoff-James orthogonality: A Minimax Theorem approach}
		
		
		\author[H. Stankovi\'c]{Hranislav Stankovi\'c}
		\address{Faculty of Electronic Engineering, University of Ni\v s, Aleksandra Medvedeva 14, Ni\v s, Serbia
		}
		\email{\url{hranislav.stankovic@elfak.ni.ac.rs}}
		
		\date{\today}
		
		\maketitle
		
		\begin{abstract}
			In this paper, we present an elementary proof of the Bhatia-\v Semrl Theorem, utilizing  the Minimax Theorem for bounded linear operators by Asplund and Ptak \cite{AsplundPtak71}. Some related results are also discussed.
		\end{abstract}
		
		\bigskip 
		\section{Introduction}
		
		Orthogonality plays a crucial role in understanding the structure and geometry of Banach spaces and the behavior of operators acting on them. While in Hilbert spaces, orthogonality is naturally defined through inner products, Banach spaces require more generalized concepts to describe relationships between vectors or operators. One such concepts is Birkhoff-James orthogonality (see \cite{Birkhoff35, James45, James47a, James47b}), which broadens the traditional notion of orthogonality.
		
		A vector $x$ in a Banach space $\mathcal{X}$ over a field $\mathbb{K}$ is said to be \emph{Birkhoff-James orthogonal} to another vector $y\in \mathcal{X}$ if for all $\lambda\in\mathbb{K}$, 
		\begin{equation*}
			\norm{x+\lambda y}\geq \norm{x}.
		\end{equation*}
		We write $x\perp_B y$, or simply, $x\perp y$. In the case of Hilbert spaces, it is easy to see that Birkhoff-James orthogonality is equivalent to the classical orthogonality. 
		
		In recent years, the study of Birkhoff-James orthogonality and the complex problems it presents in Banach spaces has gained significant attention. The issue of symmetric points with respect to this form of orthogonality has been explored in works such as \cite{CvetkovicIlicStankovic25, GhoshSainPaul16, Sain17, Turnsek05, Turnsek17}, while topics like smoothness and the norm-attainment problem have been investigated in \cite{PaulSainGhosh16, SainPaulMandal19}.
		
		One particularly intriguing case arises when the Banach space of interest is taken to be  $\mathfrak{B}(\mathcal{X})$, the algebra of all bounded linear operators on Banach or Hilbert space $\mathcal{X}$. A central result in this direction is the Bhatia-\v Semrl Theorem \cite[Theorem 1.1 and Remark 3.1]{BhatiaSemrl99}, which provides an exact characterization of Birkhoff-James orthogonality for bounded linear operators on Hilbert spaces. More precisely, the theorem states that if $A,B\in\mathfrak{B}(\mathcal{H})$, where $\mathcal{H}$ is  a  
		Hilbert space with dimension strictly greater than one, then  $A\perp B$ if and only if there exists a sequence $\{x_n\}_{n\in\mathbb{N}}$ of unit vectors in $\mathcal{H}$ such that $\lim\limits_{n\to\infty}\norm{Ax_n}=\norm{A}$ and $\lim\limits_{n\to\infty}\skal{Ax_n}{Bx_n}=0$. In finite-dimensional case, the latter condition simplifies to the existence of a unit vector $x\in\mathcal{H}$  such that $\norm{Ax}=\norm{A}$ and $\skal{Ax}{Bx}=0$. For an alternative characterization of orthogonality in the context of matrices, see \cite{LiSchneider02}.
		
		Over the years, various methods have been applied to prove the Bhatia-Šemrl Theorem. In the finite-dimensional case, proofs have utilized directional orthogonality \cite{RoyBagchiSain22} and convex optimization techniques \cite{BhattacharyyakeyGrover13}. For the real case, the orthogonality of real bilinear forms was used in \cite{RoySenapatiSain21}. In the general setting, approaches have included $r$-orthogonality \cite{Turnsek17} and the $\rho$-G\^ ateaux derivative method \cite{Keckic05}.
		\medskip 
		
		The goal of this paper is to provide an elementary proof of the Bhatia-\v Semrl Theorem in the general case. A key tool in our approach is the following result by Asplund and Ptak \cite{AsplundPtak71}.
		
		\begin{theorem}[Minimax Theorem]\cite[Theorem 3.1]{AsplundPtak71}\label{thm:minimax}
			Let $\mathcal{H},\mathcal{K}$ be at least two-dimensional inner-product spaces over a field $\mathbb{K}\in\{\mathbb{C}, \mathbb{R}\}$, and let $A,B\in\mathfrak{B}(\mathcal{H},\mathcal{K})$. Then,
			\begin{equation}\label{eq:minimax}
				\sup_{\norm{x}\leq 1}\inf_{\lambda\in\mathbb{K}}\norm{Ax+\lambda Bx}=\inf_{\lambda\in\mathbb{K}}\sup_{\norm{x}\leq 1}\norm{Ax+\lambda Bx}.
			\end{equation}
		\end{theorem}
		
		\begin{remark}
			It is easy to check that the condition \eqref{eq:minimax} is equivalent with
		\begin{equation*}
			\sup_{\norm{x}= 1}\inf_{\lambda\in\mathbb{K}}\norm{Ax+\lambda Bx}=\inf_{\lambda\in\mathbb{K}}\sup_{\norm{x}= 1}\norm{Ax+\lambda Bx}.
		\end{equation*}
		\end{remark}

		\bigskip 
		\section{Proof of Bhatia-\v Semrl Theorem}
		
		We start by proving the general version of the Bhatia-\v Semrl theorem. 
		
		\begin{theorem}\label{thm:bhatia_semrl_general}
			Let $\mathcal{H}$ be a complex Hilbert space, $\dim\mathcal{H}\geq 2$, and let $A,B\in\mathfrak{B}(\mathcal{H})$. Then $A\perp B$ if and only if there exists a sequence $\{x_n\}_{n\in\mathbb{N}}$ of unit vectors in $\mathcal{H}$ such that $\lim\limits_{n\to\infty}\norm{Ax_n}=\norm{A}$ and $\lim\limits_{n\to\infty}\skal{Ax_n}{Bx_n}=0$.
		\end{theorem}
		
		\begin{proof}
			Let $\lambda\in\mathbb{C}$ be arbitrary and let $\{x_n\}_{n\in\mathbb{N}}$ be a sequence with the mentioned properties. Then,
			\begin{align*}
				\norm{A+\lambda B}^2&\geq \norm{(A+\lambda B)x_n}^2\\
				&= \norm{Ax_n}^2+\aps{\lambda}^2\norm{Bx_n}^2+2\,\mathrm{Re\,}(\overline{\lambda}\skal{Ax_n}{Bx_n})\\
				&\geq \norm{Ax_n}^2+2\,\mathrm{Re\,}(\overline{\lambda}\skal{Ax_n}{Bx_n}).
			\end{align*}
			By letting $n\to\infty$, we get
			\begin{align*}
				\norm{A+\lambda B}^2&\geq \lim\limits_{n\to\infty}(\norm{Ax_n}^2+2\,\mathrm{Re\,}(\overline{\lambda}\skal{Ax_n}{Bx_n}))\\
				&=\lim\limits_{n\to\infty}\norm{Ax_n}^2+2\,\mathrm{Re\,}(\overline{\lambda}\lim\limits_{n\to\infty}\skal{Ax_n}{Bx_n})\\
				&=\norm{A}^2.
			\end{align*}
			Thus, $\norm{A+\lambda B}\geq\norm{A}$ for all $\lambda\in\mathbb{C}$, or equivalently, $A\perp B$. \par 
			\medskip 
			
			Now let $A\perp B$. Without loss of generality, we may assume that $\norm{A}\geq 1$. Then,
			\begin{equation*}
				\inf_{\lambda\in\mathbb{C}}\norm{A+\lambda B}=\norm{A}.
			\end{equation*}
			By Theorem \ref{thm:minimax}, we have that
			\begin{equation*}
				\sup_{\norm{x}=1}\inf_{\lambda\in\mathbb{C}}\norm{Ax+\lambda Bx}=\inf_{\lambda\in\mathbb{C}}\norm{A+\lambda B},
			\end{equation*}
			and thus
			\begin{equation}\label{eq:sup_inf_eq_A}
				\sup_{\norm{x}= 1}\inf_{\lambda\in\mathbb{C}}\norm{Ax+\lambda Bx}=\norm{A}.
			\end{equation}
			From here, we deduce that for each $n\in\mathbb{N}$ there exists $x_n\in\mathcal{H}$, $\norm{x_n}= 1$, such that
			\begin{equation}\label{eq:compact_form}
				\inf_{\lambda\in\mathbb{C}}\norm{Ax_n+\lambda Bx_n}>\norm{A}-\frac{1}{n}.
			\end{equation}
			Since $\norm{A}\geq\frac{1}{n}$, $n\in\mathbb{N}$, this is equivalent with the fact that there exists a sequence $\{x_n\}_{n\in\mathbb{N}}$ of unit vectors in $\mathcal{H}$, such that for all $\lambda\in\mathbb{C}$ and all $n\in\mathbb{N}$,
			\begin{equation}\label{eq:expanded_form}
				\norm{A}^2-\frac{2}{n}\norm{A}+\frac{1}{n^2}<	\norm{Ax_n}^2+\aps{\lambda}^2\norm{Bx_n}^2+2\,\mathrm{Re\,}(\overline{\lambda}\skal{Ax_n}{Bx_n}).
			\end{equation}
			
			It follows from \eqref{eq:compact_form} that
			\begin{equation*}
				\norm{A}-\frac{1}{n}<\inf_{\lambda\in\mathbb{C}}\norm{Ax_n+\lambda Bx_n}\leq \norm{Ax_n+0\cdot Bx_n}=\norm{Ax_n}\leq \norm{A}.
			\end{equation*}
			By letting $n\to\infty$, we immediately obtain that $\lim\limits_{n\to\infty}\norm{Ax_n}=\norm{A}$. 
			
			Furthermore,   \eqref{eq:expanded_form} yields
			\begin{equation*}
				\norm{A}^2-\frac{2}{n}\norm{A}+\frac{1}{n^2}< \norm{A}^2+\aps{\lambda}^2\norm{B}^2+2\,\mathrm{Re\,}(\overline{\lambda}\skal{Ax_n}{Bx_n}),\quad \lambda\in\mathbb{C},\, n\in\mathbb{N},
			\end{equation*}
			and thus,
			\begin{equation}\label{eq:simplified_form}
				-\frac{2}{n}\norm{A}+\frac{1}{n^2}< \aps{\lambda}^2\norm{B}^2+2\,\mathrm{Re\,}(\overline{\lambda}\skal{Ax_n}{Bx_n}),\quad \lambda\in\mathbb{C},\, n\in\mathbb{N}.
			\end{equation}
			Since $\{\skal{Ax_n}{Bx_n}\}_{n\in\mathbb{N}}$ is a bounded sequence of complex numbers, without loss of generality we may assume that $\lim\limits_{n\to\infty}\skal{Ax_n}{Bx_n}$ exists (otherwise, we can choose an appropriate subsequence). Set $L:=\lim\limits_{n\to\infty}\skal{Ax_n}{Bx_n}$. By letting $n\to\infty$ in \eqref{eq:simplified_form}, we obtain
			\begin{equation}\label{eq:L_all_lambda}
				0\leq \aps{\lambda}^2\norm{B}^2+2\,\mathrm{Re\,}(\overline{\lambda} L),\quad \lambda\in\mathbb{C}.
			\end{equation}
			
			By taking $\lambda>0$ and $\lambda<0$ in \eqref{eq:L_all_lambda}, we get
			\begin{equation*}
				0\leq \lambda\norm{B}^2+2\,\mathrm{Re\,}L
			\end{equation*}
			and
			\begin{equation*}
				0\geq \lambda\norm{B}^2+2\,\mathrm{Re\,}L,
			\end{equation*}
			respectively. By letting $\lambda\to 0$ in the previous two inequalities, it immediately follows that $\mathrm{Re\,}L=0$.
			
			By considering $i\lambda$ instead of $\lambda$ in \eqref{eq:L_all_lambda} and using the fact that $\mathrm{Re\,}(iz)=-\mathrm{Im\,}z$ for any $z\in\mathbb{C}$, we may also deduce that $\mathrm{Im\,}L=0$. Therefore, $L=0$, as desired. 
		\end{proof}

		\begin{remark}
			We can easily see that the previous proof works perfectly well in the case of real Hilbert spaces (or just inner-product spaces). Moreover, by omitting the final paragraph of the proof, we can derive \cite[Theorem 2.4]{Turnsek17} concerning $r$-orthogonality. 
		\end{remark}

		Although the proof in a finite-dimensional case can be easily derived as a corollary of Theorem \ref{thm:bhatia_semrl_general}, here we  also present an elementary direct proof using again Theorem \ref{thm:minimax} with a somewhat different argumentation. 
		
		\begin{theorem}\label{thm:bhatia_semrl_finite}
			Let $\mathcal{H}$ be a finite-dimensional complex Hilbert space, $\dim\mathcal{H}\geq 2$, and let $A,B\in\mathfrak{B}(\mathcal{H})$. Then $A\perp B$ if and only if there exists $x\in\mathcal{H}$, $\norm{x}=1$, such that $\norm{Ax}=\norm{A}$ and $\skal{Ax}{Bx}=0$. 
		\end{theorem}
		
		\begin{proof}
			If such a vector exists, then
			
			\begin{align*}
				\norm{A+\lambda B}^2&\geq \norm{(A+\lambda B)x}^2\\
				&= \norm{Ax}^2+\aps{\lambda}^2\norm{Bx}^2\\
				&\geq \norm{A}^2,
			\end{align*}
			and thus, $A\perp B$.\par 
			\medskip 
			
			Conversely, if $A\perp B$, we have that \eqref{eq:sup_inf_eq_A} holds. Hence, there exists a sequence $\{x_n\}_{n\in\mathbb{N}}$ of unit vectors in $\mathcal{H}$ such that
			\begin{equation}\label{eq:lim_inf}
				\lim\limits_{n\to\infty}\inf_{\lambda\in\overline{\mathbb{D}}}\norm{(A+\lambda B)x_n}=\norm{A},
			\end{equation}
			where $\overline{\mathbb{D}}$ denotes the closed unit disk in $\mathbb{C}$. Since $\mathcal{H}$ is finite-dimensional, without loss of generality, we may assume that $\lim\limits_{n\to\infty}x_n=x$, for some $x\in\mathcal{H}$, $\norm{x}=1$.
			
			For each $n\in\mathbb{N}$, let $f_n: \overline{\mathbb{D}}\to \mathbb{R}$ be defined as
			\begin{equation*}
				f_n(\lambda)= \norm{(A+\lambda B)x_n},\quad \lambda\in \overline{\mathbb{D}}.
			\end{equation*}
			It is easy to check that the sequence $\{f_n\}_{n\in\mathbb{N}}$ uniformly converges to $f$ on $\overline{\mathbb{D}}$, where $$f(\lambda)=\norm{(A+\lambda B)x},\quad  \lambda\in \overline{\mathbb{D}}.$$
			Thus, the limit and infimum can switch places in \eqref{eq:lim_inf}, i.e.,
			\begin{equation}\label{eq:inf_x}
				\inf_{\lambda\in\overline{\mathbb{D}}}\norm{(A+\lambda B)x}=\norm{A}. 
			\end{equation}
			
			The conclusion now easily follows. Indeed, from 
			\begin{equation*}
				\norm{A}=\inf_{\lambda\in\overline{\mathbb{D}}}\norm{(A+\lambda B)x}\leq \norm{Ax}\leq \norm{A},
			\end{equation*}
			we have that $\norm{Ax}=\norm{A}$. Equality \eqref{eq:inf_x} now implies that for each $\lambda\in\overline{\mathbb{D}}$, 
			\begin{equation*}
				\norm{Ax+\lambda Bx}\geq \norm{Ax},
			\end{equation*}
			which is equivalent with the fact that $\skal{Ax}{Bx}=0$. This completes the proof. 
		\end{proof}

		\bigskip 
		\section*{Declarations}
		\noindent{\bf{Funding}}\\
		This work has been supported by the Ministry of Science, Technological Development and Innovation of the Republic of Serbia [Grant Number: 451-03-137/2025-03/ 200102].			
		
		\vspace{0.5cm}
		
		\noindent{\bf{Availability of data and materials}}\\
		\noindent No data were used to support this study.
		\vspace{0.5cm}\\
		\noindent{\bf{Competing interests}}\\
		\noindent The author declares that he has no competing interests.
		\vspace{0.5cm}
		
		
		\vspace{0.5cm}
		
		
		
	\end{document}